\documentclass[reqno,12pt]{amsart}
\usepackage{amsthm,amssymb,amsmath, tikz}
\usepackage[pdftex]{hyperref}
\usepackage[all]{xy}
\theoremstyle{definition}
\newtheorem{defi}{Definition}[section]
\newtheorem{``defi''}{``Definition''}[section]
\theoremstyle{remark}
\newtheorem{exam}[defi]{Example}
\newtheorem{exams}[defi]{Examples}
\theoremstyle{plain}
\newtheorem{thm}[defi]{Theorem}
\newtheorem{lem}[defi]{Lemma}
\newtheorem{cor}[defi]{Corollary}

\newtheorem{rem}[defi]{Remark}

\newcommand{\Z}{\mathbb{Z}}
\newcommand{\N}{\mathbb{N}}

\begin{document}

\title{Examples of $*$-commuting maps}
\author[B. Maloney and P. N. Willis]{Ben Maloney and Paulette N. Willis}

\address{Department of Mathematics \\ University of Houston \\ Houston, TX 77204-3008 \\ USA}
\email{pnwillis@math.uh.edu}

\thanks{This research was partially supported by NSF Mathematical Sciences Postdoctoral Fellowship DMS-1004675, the University of Iowa Graduate College Fellowship as part of the Sloan Foundation Graduate Scholarship Program, and the University of Iowa Department of Mathematics NSF VIGRE grant DMS-0602242.}

\date{\today}



\begin{abstract}
We introduce the concept of a $1$-coaligned $k$-graph and prove that the shift maps of a $k$-graph pairwise $*$-commute if and only if the $k$-graph is $1$-coaligned.  We then prove that for $2$-graphs $\Lambda$ generated from basic data $*$-commuting shift maps is equivalent to a condition that implies that $C^*(\Lambda)$ is simple and purely infinite. We then consider full shift spaces and introduce a condition on a block map which ensures the associated sliding block code $*$-commutes with the shift.
\end{abstract}

\numberwithin{equation}{section}
\maketitle

\section{Introduction}
Suppose $X$ is a set and $S,T: X \to X$  are commuting functions.  We say that $S$ and $T$ $*$-commute if for every $(y,z) \in X \times X$ such that $S(y)=T(z)$, there exists a unique $x \in X$ such that $T(x)=y$ and $S(x)=z$.  The concept of $*$-commuting maps was first introduced in \cite[\S 5]{AR}, where Arzumanian and Renault studied $*$-commuting pairs of local homeomorphisms on a compact space $X$.  In \cite{ER07} Exel and Renault expand on this idea and provide many interesting examples.  Additional examples of $*$-commuting maps are in \cite{pW10}.

Higher-rank graphs (or $k$-graphs) $\Lambda$ were introduced by Kumjian and Pask \cite{KP00} to provide combinatorial models for the higher-rank Cuntz-Krieger $C^*$-algebras of Robertson and Steger \cite{RS99}.  In Section~\ref{kgraph} we prove that the shift maps on the infinite path space of $\Lambda$ pairwise $*$-commute if and only if the $k$-graph is $1$-coaligned in the sense that for each pair of edges $(e,f)$ with the same source there exists a pair of edges $(g,h)$ such that $ge=hf$.  This equivalence is purely set-theoretic: the maps are not required to be continuous in any sense.

In Section~\ref{BD2graph} we restrict ourselves to the $2$-graphs discussed in \cite{PRW09}, which model higher-dimensional subshifts.  The motivating example in \cite{PRW09} is a subshift introduced by Ledrappier \cite{Led78}, which is also one of Exel and Renault's main examples in \cite{ER07}.  In \cite[\S 11]{ER07} Exel and Renault prove that the 2 shift maps on Ledrappier's dynamical system $*$-commute.  We discover conditions on a graph from \cite{PRW09} which ensure that graph is $1$-coaligned. These conditions also ensure that $C^*(\Lambda)$ is simple and purely infinite.

In Section~\ref{FSS} we consider full shift spaces.  Morphisms between shift spaces are called sliding block codes, and any such morphism $\tau_d$ is built from a block map $d : A^n \to A$  (see \cite{Hed69, LM95}).   We give a condition on the block map which ensures that the sliding block code and the shift $*$-commute.

\section{$k$-graphs whose shifts $*$-commute} \label{kgraph}

A \emph{k-graph} is a pair $(\Lambda, d)$ consisting of a countable category $\Lambda$ and a functor $d: \Lambda \to \mathbb{N}^k$, called the \emph{degree map}, satisfying the \emph{factorization property}: for every $\lambda \in \Lambda$ and $m,n \in \mathbb{N}^k$ with $d(\lambda)= m+n$, there are unique elements $\mu, \nu \in \Lambda$ such that $d(\mu)=m$, $d(\nu)=n$ and $\lambda= \mu \nu$.  For $k\geq 1$, $\Omega_{k}$ is a category with unit space $\Omega_{k}^{0}=\mathbb{N}^{k}$, morphism space $\Omega_{k}^{*}=\{(m,n)\in\mathbb{N}^{k}\times\mathbb{N}^{k}:m\leq n\}$, range map $r(m,n)=m$, and source map $s(m,n)=n$. Let $d:\Omega\rightarrow\mathbb{N}^{k}$ be defined by $d(m,n)=m-n$, then $(\Omega_{k},d)$ is a $k$-graph, which we denote by $\Omega_{k}$.  A morphism between $k$-graphs $(\Lambda_1, d_1)$ and $(\Lambda_2, d_2)$ is a functor $f:\Lambda_1 \to \Lambda_2$ compatible with the degree maps.  For a $k$-graph $(\Lambda,d)$ 
\[
\Lambda^{\infty}:=\{x:\Omega_{k}\rightarrow\Lambda:x\mbox{ is a }k\mbox{-graph morphism}\}
\] 
is the \textit{infinite path space} of $\Lambda$ and for $n \in \mathbb{N}^k$ $\Lambda^n$ denotes the set of all paths of degree $n$.

We shall use $\mathbf{0}$ to denote the zero vector in $\mathbb{N}^k$.  We denote sections of these paths with range $m\in\mathbb{N}^{k}$ and source $n\in\mathbb{N}^{k}$ by $x(m,n)$.  The paths $x(m,m)$ have degree $\mathbf{0}$ hence are vertices.  In the literature it is common to write $x(m):=x(m,m)$, but as in \cite{PRW09} we refrain from doing this as $x(m)$ has another meaning.  For $p\in\mathbb{N}^{k}$, we define the map $\sigma^{p}:\Lambda^{\infty}\rightarrow\Lambda^{\infty}$ by $\sigma^{p}(x)(m,n)=x(m+p,n+p)$ for $(m,n)\in\Omega_{k}^{*}$.  Observe that $\sigma^p \sigma^q= \sigma^{p+q}$.

\begin{rem} \label{Props-Lambda-infty}
In this remark we state properties of $\Lambda^{\infty}$ that will be useful in this paper.  These properties have been discussed in several papers including \cite{KP00, PRW09}.  
\begin{enumerate}
    \item \label{P1} For all $\lambda\in\Lambda$ and $x\in\Lambda^{\infty}$ with $x(\mathbf{0}, \mathbf{0})=s(\lambda)$, there is a unique $\lambda x$ such that $x=\sigma^{d(\lambda)}(\lambda x)$ and $\lambda= (\lambda x)(\mathbf{0},d(\lambda))$. (\cite[Proposition 2.3]{KP00})
    \item  \label{P2} For every $x \in \Lambda^{\infty}$ and $p\in\mathbb{N}^{k}$, $x=x(\mathbf{0},p)\sigma^{p}x$. (\cite[Proposition 2.3]{KP00})
    \item \label{P3} For all $v \in \Lambda^0$ there exists $x \in \Lambda^{\infty}$ with $r(x)=v$.  This condition says that we assume the $k$-graph has no sources. (\cite[Definition 1.4]{KP00})
    \item \label{P4} For all $\lambda, \mu \in \Lambda$ and $x\in\Lambda^{\infty}$ such that $r(x)=s(\mu)$ and $r(\mu)=s(\lambda)$ we have $\lambda(\mu x)= (\lambda \mu) x$.
  \end{enumerate}
\end{rem}

\begin{defi}
A $k$-graph $\Lambda$ is $1$-\textit{coaligned} if for all $1 \leq i,j \leq k$ such that $i \neq j$, and $(e^i,e^j)\in\Lambda^{e_i}\times\Lambda^{e_j}$ with $s(e^i)=s(e^j)$ there exists a unique pair $(f^i,f^j)\in\Lambda^{e_i}\times\Lambda^{e_j}$ such that $f^ie^j=f^je^i$.
\end{defi}

\begin{thm} \label{thm:SC-classify}
Suppose $(\Lambda,d)$ is a $k$-graph with no sources, and $\sigma^{e_{i}}:\Lambda^{\infty}\rightarrow\Lambda^{\infty}$ is the one-sided shift in the $e_i$ direction.  The maps $\sigma^{e_i},\sigma^{e_j}$ $*$-commute for every $i \neq j$  if and only if $(\Lambda,d)$ is $1$-coaligned.
\end{thm}

\begin{proof}
Suppose $\sigma^{e_i},\sigma^{e_j}$ $*$-commute for $i \neq j$ and fix $(e^i,e^j)\in\Lambda^{e_i}\times\Lambda^{e_j}$ with $s(e^i)=s(e^j)$. By Remark~\ref{Props-Lambda-infty}~\eqref{P3} there exists $w \in \Lambda^{\infty}$ with $r(w)= s(e^i)$.  Define $y:=e^iw$ and $z:=e^jw$. Observe that by Remark~\ref{Props-Lambda-infty}~\eqref{P1} $\sigma^{e_i}(y)=w=\sigma^{e_j}(z)$. Since $\sigma^{e_i}$ and $\sigma^{e_j}$ $*$-commute there exists a unique $x\in\Lambda^{\infty}$ such that $\sigma^{e_i}(x)=z$ and $\sigma^{e_j}(x)=y$. Define $f^i:=x(\mathbf{0},e_i)$ and $f^j:=x(\mathbf{0},e_j)$. So Remark~\ref{Props-Lambda-infty}~\eqref{P2} gives us 
\begin{equation*}
x=x(\mathbf{0},e_i)\sigma^{e_i}(x)=f^iz=f^ie^jw\mbox{ and }x=x(\mathbf{0},e_j)\sigma^{e_j}(x)=f^jy=f^je^iw.
\end{equation*}
Hence $(f^i,f^j)\in\Lambda^{e_i}\times\Lambda^{e_j}$ satisfies $f^ie^j=f^je^i$.  To demonstrate uniqueness suppose there exists a pair $(g^i,g^j)\in\Lambda^{e_i}\times\Lambda^{e_j}$ such that $g^ie^j=g^je^i$. Then $g^ie^jw=g^je^iw$ and $\sigma^{e_i}(g^ie^jw)=e^jw=z$ and $\sigma^{e_j}(g^je^iw)=e^iw=y$. Since $\sigma^{e_i},\sigma^{e_j}$ $*$-commute, $g^ie^jw=x=g^je^iw$. By Remark~\ref{Props-Lambda-infty}~\eqref{P1}, $g^i=x(\mathbf{0},e_i)=f^i$ and $g^j=x(\mathbf{0},e_j)=f^j$.

Suppose that $(\Lambda, d)$ is $1$-coaligned. We see that $\sigma^{e_{i}}\sigma^{e_{j}}=\sigma^{e_{j}}\sigma^{e_{i}}$ follows from the fact that $\sigma^p \sigma^q= \sigma^{p+q}$ and $\mathbb{N}^k$ is commutative. Suppose $y,z\in\Lambda^{\infty}$ such that $\sigma^{e_i}(y)=\sigma^{e_j}(z)=w$, say. Define $e^i := y(\mathbf{0},e_i)$ and $e^j := z(\mathbf{0},e_j)$, which is a pair in $\Lambda^{e_i}\times\Lambda^{e_j}$.  By Remark~\ref{Props-Lambda-infty}~\eqref{P2} we have
\begin{equation*}
y=y(\mathbf{0}, e_i) \sigma^{e_i}(y)= e^i w \mbox { and } z=z(\mathbf{0}, e_j) \sigma^{e_j}(z)= e^j w.
\end{equation*}
Therefore $s(e^i)=r(w)=s(e^j)$.  Since $(\Lambda, d)$ is $1$-coaligned there exists a unique pair $(f^i,f^j) \in \Lambda^{e_i} \times \Lambda^{e_j}$ such that $f^ie^j=f^je^i$. Define $x:=f^je^iw$ (which equivalently equals $f^ie^jw$), then $\sigma^{e_i}(x)=\sigma^{e_i}(f^ie^jw)=e^jw=z$ and $\sigma^{e_j}(x)=\sigma^{e_j}(f^je^iw)=e^iw=y$. The uniqueness of $x$ follows from the uniqueness of the pair $(f^i, f^j)$.  Hence $\sigma^{e_i},\sigma^{e_j}$ *-commute.
\end{proof}

Observe that Theorem~\ref{thm:SC-classify} is independent of the topology of the path space.  Pairwise $*$-commuting shifts being equivalent to the $k$-graph being $1$-coaligned is a set-theoretic property.   Next we discuss the usual topology put on the path space of a $k$-graph and discuss our results in terms of that topology.

For a $k$-graph $\Lambda$ define 
\[ Z(\lambda)=\{x\in\Lambda^{\infty}:x(\mathbf{0},d(\lambda))=\lambda\} \]
called the \emph{cylinder set of $\lambda$}.   The sets $Z(\lambda)$ are compact and form a subbasis for a locally compact Hausdorff topology on $\Lambda^{\infty}$ (see \cite[Proposition 2.8]{KP00}).

\begin{cor}
If $\Lambda$ is a $1$-coaligned row-finite $k$-graph with no sinks or sources then for each $i \neq j$, $\sigma^{e_i}$ and $\sigma^{e_j}$ are $*$-commuting surjective local homeomorphisms.
\end{cor}

\begin{proof}
By Theorem~\ref{thm:SC-classify} $\sigma^{e_i}$ and $\sigma^{e_j}$ $*$-commute.  To see that $\sigma^{e_i}$ is surjective fix $y\in\Lambda^{\infty}$.  Since $\Lambda$ has no sinks, there exists $e\in\Lambda^{e_{i}}r(y)$ such that $s(e)=r(y)$.  By Remark~\ref{Props-Lambda-infty}~\eqref{P1} $y= \sigma^{e_i}(ey)$, therefore $\sigma^{e_i}$ is surjective.  One can show that $\sigma^{e_i}$ is continuous by showing $(\sigma^{e_{i}})^{-1}Z(\lambda)=\bigsqcup_{e\in\Lambda^{e_{i}}r(\lambda)}Z(e\lambda)$.  We see that $\sigma^{e_i}$ is a local homeomorphism by showing that for $x \in \Lambda^{\infty}$ $\sigma^{e_{i}}:Z(x(\mathbf{0}, e_i))\rightarrow Z(s(x(\mathbf{0}, e_i)))$  is a homeomorphism.
\end{proof}

\section{$2$-graphs from Basic Data} \label{BD2graph}

We begin this section by briefly reviewing the process of constructing a $2$-graph from basic data (see \cite[\S 3]{PRW09}).  A subset $T$ of $\N^2$ is \emph{hereditary} if for $j\in T$, each $i$ such that $\mathbf{0} \leq i\leq j$, $i\in T$.  There are four variables that make up \emph{basic data}:
\begin{itemize}
\item a finite hereditary subset of $\N^2$ called the \emph{tile} denoted $T$,
\item an \emph{alphabet} $\{0,\ldots,q-1\}$ identified with $\mathbb{Z}/q \mathbb{Z}$,
\item an element $t$ of the alphabet called the \emph{trace},
\item a weight function $w:T\rightarrow\{0,\ldots,q-1\}$ called the \emph{rule}.
\end{itemize}  
The vertex set of $\Lambda(T,q,t,w)$ is
\begin{equation}
\label{eq:trace}
\Lambda^{0}:= \Big \{v:T\rightarrow\mathbb{Z}/q \mathbb{Z} \Big| \sum_{i\in T}w(i)v(i)\equiv t\pmod{q} \Big \}.
\end{equation}

For $S\subset\Z^2$ and $n\in\Z^2$, define the translate of $S$ by $n$ by $S+n:=\{i+n:i\in S\}$. Set $T(n):=\bigcup_{\mathbf{0} \leq m\leq n}T+m$. If $f:S\rightarrow \mathbb{Z}/q \mathbb{Z}$ is defined on a subset $S$ of $\N^2$ containing $T+n$, then we define $f|_{T+n}:T\rightarrow\mathbb{Z}/q \mathbb{Z}$ by $f_{T+n}(i)=f(i+n)$ for $i\in T$. A \textit{path of degree $n$} is a function $\lambda:T(n)\rightarrow\mathbb{Z}/q \mathbb{Z}$ such that $\lambda|_{T+m}= \lambda(m,m)$ is a vertex for $\mathbf{0} \leq m\leq n$, with source $s(\lambda)=\lambda|_{T+n}$ and range $r(\lambda)=\lambda|_{T}$.  For $\lambda \in \Lambda^p$ and $\mathbf{0} \leq m \leq n \leq p$, the segment $\lambda(m,n)$ is the path of degree $n-m$ defined by 
\[  \lambda(m,n)(i)= \lambda(m+i) \mbox { for } i \in T(n-m). \]

Suppose $w$ has invertible corners, in the sense that for $(c_1,c_2):=\bigvee\{i:i\in T\}$ $w(c_1e_1)$ and $w(c_2e_2)$ are invertible elements of $\mathbb{Z}/q \mathbb{Z}$, and $\mu \in \Lambda^m$, $\nu  \in \Lambda^n$ such that $s(\mu)=r(\nu)$.  Then there is a unique path $\lambda  \in \Lambda^{m+n}$ such that $\lambda(\mathbf{0}, m)=\mu$ and $\lambda(m, m+n)= \nu$ (\cite[Proposition 3.2]{PRW09}).  By defining composition using the unique path $\lambda$ and defining a degree map Pask, Raeburn and Weaver \cite[Theorem 3.4]{PRW09} prove that there exists a unique 2-graph which we denote $\Lambda(T,q,t,w)$.  We assume that $c_1, c_2 > 0$.


\begin{thm}
\label{thm:coalign-invert}
Suppose $(T,q,t,w)$ is basic data with invertible corners. Then $\Lambda(T,q,t,w)$ is $1$-coaligned if and only if $w(\mathbf{0})$ is invertible in $\mathbb{Z}/q \mathbb{Z}$.
\end{thm}

\begin{lem} \label{Lem-for-Thm}
Suppose $e^b \in \Lambda^{e_1}$ and  $e^r \in \Lambda^{e_2}$ are edges in $\Lambda(T,q,t,w)$ such that $s(e^b)=s(e^r)$.
\begin{enumerate}
  \item \label{Lem-unique} If $\mu$ is a path with $d(\mu)=e_1 + e_2$ such that $\mu (e_1, e_1 + e_2)= e^r$ and $\mu (e_2, e_1 + e_2)= e^b$, then
\begin{equation}\label{eq:mu-path}
           \mu(i)=\begin{cases}
            e^r(i-e_1) & \mbox{if } i \in (T+ e_1) \cup (T + e_1 +e_2)\\
            e^b(i-e_2) & \mbox{if } i \in (T+ e_2)  \cup (T + e_1 +e_2)
           \end{cases}
            \end{equation}
        and $\mu(\mathbf{0})$ satisfies 
        \begin{equation} 
        \label{eq:mu-zero}
w(\mathbf{0}) \mu(\mathbf{0})= t-\sum_{i\in T\backslash \{ \mathbf{0}\}}w(i) \mu(i).
        \end{equation}
   \item \label{Lem-exist} There is a well-defined function $\lambda: T(e_1+e_2)\backslash \{\mathbf{0}\} \to \mathbb{Z}/q \mathbb{Z}$ such that
        \begin{equation}
\label{eq:lambda-path}
\lambda(i)=\begin{cases}
            e^r(i-e_1)  &  i_1>0\\
            e^b(i-e_2)  &  i_2>0.
           \end{cases}
\end{equation}
If $t_0 \in \mathbb{Z}/q \mathbb{Z}$ satisfies     
   \begin{equation} 
        \label{eq:t-zero}
w(\mathbf{0}) t_0= t-\sum_{i\in T\backslash \{\mathbf{0} \}}w(i) \lambda(i),
        \end{equation}
and we define $\lambda(\mathbf{0})= t_0$, then $\lambda$ is a path of degree $e_1+e_2$ with $\lambda  (e_1, e_1 + e_2)= e^r$ and $\lambda (e_2, e_1 + e_2)= e^b$.  
\end{enumerate}
\end{lem}

\begin{proof}
Take $i \in (T+ e_1) \cup (T + e_1 +e_2)$.  Then $i_1 >0$, so $i-e_1 \in T \cup (T+e_2)$  Hence we have
\[ \mu(i)= \mu((i - e_1) + e_1)= \mu (e_1, e_1+e_2)(i-e_1)= e^r(i-e_1). \]
Similarly take $i \in (T+ e_2)  \cup (T + e_1 +e_2)$.  Then $i_2 >0$, so $i-e_2 \in T \cup (T+e_1)$  Hence we have
\[ \mu(i)= \mu((i - e_2) + e_2)= \mu (e_2, e_1+e_2)(i-e_2)= e^b(i-e_2). \]
Since $\mu|_T$ is a vertex Equation~\eqref{eq:trace} shows that
\[ t=  \sum_{i \in T} w(i) \mu(i)=  w(\mathbf{0}) \mu(\mathbf{0}) + \sum_{i \in T \backslash \{\mathbf{0}\}} w(i) \mu(i). \]
Therefore $\mu(\mathbf{0})$ satisfies Equation~\eqref{eq:mu-zero}.

To show that $\lambda$ is well-defined function we consider the case when $i \in T(e_1+e_2)  \backslash \{ \mathbf{0}\}$ satisfies both definitions. So suppose $i \in T + e_1+e_2$, which is when $i$ satisfies $i_1 >0$ and $i_2>0$ so both definitions apply.  Then
\begin{align*}
\lambda(i)= e^r(i-e_1) &=e^r(i-e_1-e_2+e_2)=s(e^r)(i-e_1-e_2)\\
                 &=s(e^b)(i-e_1-e_2)=e^b(i-e_2-e_1+e_1)=e^b(i-e_2)
\end{align*}
therefore $\lambda$ is well-defined.  To prove that $\lambda$ is a path of degree $e_1+e_2$ we must show that $\lambda|_{T+m}$ is a path for all $\mathbf{0} \leq m \leq e_1+e_2$.  Since $\lambda|_T= \lambda(\mathbf{0})= t_0$ satisfies Equation~\eqref{eq:t-zero} $\lambda|_T$ is a vertex. Since $\lambda|_{T+e_1}= \lambda(e_1, e_1)=r(e^r)$, $\lambda|_{T+e_2}=\lambda(e_2, e_2)=r(e^b)$, and $\lambda|_{T+e_1+e_2}=\lambda(e_1 + e_2, e_1 + e_2)=s(e^r)=s(e^b)$ $\lambda |_{T+m}$ is a vertex for every $m \in \{\mathbf{0}, e_1, e_2, e_1+e_2 \}$. Therefore $\lambda$ is a path of degree $e_1+e_2$.   Observe that
\[ \lambda(e_1, e_1+e_2)(i)= \lambda(i +e_1)= e^r(i+e_1-e_1)=e^r(i)\]
and
\[ \lambda(e_2, e_1+e_2)(i)= \lambda(i +e_2)= e^b(i+e_2-e_2)=e^b(i).\]
\end{proof}

\begin{proof}[Proof of Theorem~\ref{thm:coalign-invert}]
Suppose $w(\mathbf{0})$ is invertible in $\mathbb{Z}/q \mathbb{Z}$ and $(e^b,e^r)\in\Lambda^{e_1}\times\Lambda^{e_2}$ satisfies $s(e^b)=s(e^r)$.  Let $\lambda: T(e_1+e_2)\backslash \{\mathbf{0}\} \to \mathbb{Z}/q \mathbb{Z}$ be the function satisfying Equation~\eqref{eq:lambda-path} and 
\[ \lambda(\mathbf{0}):= t_0 := w(\mathbf{0})^{-1} \Big( t-\sum_{i\in T\backslash \{ \mathbf{0}\}}w(i) \lambda(i) \Big). \]
Then by Lemma~\ref{Lem-for-Thm}~\eqref{Lem-exist} $\lambda$ is a path of degree $e_1+e_2$ with $\lambda  (e_1, e_1 + e_2)= e^r$ and $\lambda (e_2, e_1 + e_2)= e^b$.  Define $f^{b}:=\lambda(\mathbf{0},e_1)$ and $f^{r}:= \lambda(\mathbf{0},e_2)$. Then $(f^b, f^r) \in \Lambda^{e_1}\times\Lambda^{e_2}$ and $f^{r}e^{b}=f^{b}e^{r}=\lambda$.  Suppose there exists $(g^b, g^r) \in \Lambda^{e_1}\times\Lambda^{e_2}$ and $g^{r}e^{b}=g^{b}e^{r}=\mu$, say.  To show that $f^b=g^b$ and $f^r=g^r$ it suffices to show that $\lambda=\mu$.  Observe that by Lemma~\ref{Lem-for-Thm}~\eqref{Lem-unique} for $i \neq \mathbf{0}$ $\lambda(i)=\mu(i)$ so we must show that $\lambda(\mathbf{0})=\mu(\mathbf{0})$.   Equation~\eqref{eq:mu-zero} is satisfied by $\mu(\mathbf{0})$ and since $w(\mathbf{0})$ is invertible $\mu(\mathbf{0})= w(\mathbf{0})^{-1}  \Big(t-\sum_{i\in T\backslash \{\mathbf{0}\}}w(i)v(i) \Big)= \lambda(\mathbf{0})$.

Conversely, suppose that $w(\mathbf{0})$ not invertible in $\mathbb{Z}/q \mathbb{Z}$ and $(e^b,e^r)\in\Lambda^{e_1}\times\Lambda^{e_2}$ satisfy $s(e^b)=s(e^r)$. Then either there is no $t_0$ that satisfies Equation~\eqref{eq:t-zero}, or Equation~\eqref{eq:t-zero} has more than one solution.  Suppose first that there does not exist $t_0$ that satisfies Equation~\eqref{eq:t-zero}.  Then by Lemma~\ref{Lem-for-Thm}~\eqref{Lem-unique} there does not exist a path $\mu$ of degree $e_1 +e_2$ with $\mu (e_1, e_1 + e_2)= e^r$ and $\mu (e_2, e_1 + e_2)= e^b$ because if there were then $\mu(\mathbf{0})$ would satisfy Equation~\eqref{eq:t-zero}.  Hence there does not exist a pair $(f^b,f^r)\in\Lambda^{e_1}\times\Lambda^{e_2}$ such that $f^be^r=f^re^b$.   Therefore $\Lambda$ is not $1$-coaligned.  

Next suppose there exist $t_1, t_2$ which both satisfy Equation~\eqref{eq:t-zero}.  Then applying Lemma~\ref{Lem-for-Thm}~\eqref{Lem-exist} twice there are two paths $\lambda_1, \lambda_2$ of degree $e_1+e_2$ such that $\lambda_1 (e_1, e_1 + e_2)= e^r= \lambda_2 (e_1, e_1 + e_2)$, $\lambda_1 (e_2, e_1 + e_2)= e^b=\lambda_2 (e_2, e_1 + e_2)$ and $\lambda_1(\mathbf{0})= t_1 \neq t_2= \lambda_2(\mathbf{0})$. Define $f^{b}:=\lambda_1(\mathbf{0},e_1)$, $f^{r}:= \lambda_1(\mathbf{0},e_2)$, $g^{b}:=\lambda_2(\mathbf{0},e_1)$, and $g^{r}:= \lambda_2(\mathbf{0},e_2)$. Then $(f^b, f^r), (g^b, g^r) \in \Lambda^{e_1}\times\Lambda^{e_2}$ such that$f^{r}e^{b}=f^{b}e^{r}$ and $g^{r}e^{b}=g^{b}e^{r}$.  Observe that $f^r \neq g^r$ since $f^r(\mathbf{0})= \lambda_1(\mathbf{0}) \neq  \lambda_2(\mathbf{0})=g^r(\mathbf{0})$. Therefore $\Lambda$ is not $1$-coaligned.
\end{proof}

\begin{defi}
Let $(T,q,t,w)$ be basic data.  The rule \emph{$w$ has three invertible corners} if $w(\mathbf{0})$, $w(c_1e_1)$, and $w(c_2e_2)$ are all invertible in $\Z/q\Z$ (implicitly demanding that $c_1 \geq 1$ and $c_2 \geq 1$).
\end{defi}

\begin{cor} \label{TFAE}
Let $\Lambda(T,q,t,w)$ be a $2$-graph.  Then the following are equivalent:
\begin{enumerate}
  \item  the maps $\sigma^{e_1}$ and $\sigma^{e_2}$ $*$-commute,
  \item $\Lambda(T,q,t,w)$ is $1$-coaligned,
  \item the rule $w$ has three invertible corners.
\end{enumerate}
\end{cor}

\begin{proof}
Theorem~\ref{thm:SC-classify} shows (1) if and only if (2).  Theorem~\ref{thm:coalign-invert} shows (2) if and only if (3).
\end{proof}

\begin{exams}
\begin{enumerate}
\item Given basic data $(T,q,t,w)$ and invertible corners, whenever $w(\mathbf{0})=1$, the shift maps will $*$-commute.
\item Given any alphabet of prime cardinality $q$, for any rule with $w(\mathbf{0}) \neq 0$, the shift maps will $*$-commute.
\end{enumerate}
\end{exams}

\begin{rem}
It is curious that the condition for $*$-commuting is exactly the condition in  \cite[Theorem 6.1]{PRW09} for $C^*(\Lambda)$ to be simple and purely infinite.
\end{rem}

\section{Full shift spaces} \label{FSS}

Let $A$ be a finite alphabet.  Let $A^n$ denote the words of length $n$, let $A^* :=\bigcup_{n \geq 1} A^n$, and let $A^{\mathbb{N}}$ denote the one-sided infinite sequence space of elements in $A$, which is compact by Tychonoff's Theorem.  The cylinder sets $Z(\mu) := \{ x \in A^{\mathbb{N}}: x_1 \cdots x_{|\mu|}= \mu \}$ for $\mu \in A^*$ form a basis of clopen sets.  Let $\sigma:A^{\mathbb{N}} \to A^{\mathbb{N}}$ defined by $\sigma(x_1 x_2 x_3 \cdots)= x_2 x_3 \cdots$ be the shift map.

A \emph{block map} is a function $d: A^n \to A$ for some $n \in \mathbb{N}$.  For any block map $d$ we define $\tau_d: A^{\mathbb{N}} \to A^{\mathbb{N}}$ by $\tau_d(x)_i= d(x_i \cdots x_{i+n-1})$.  We call $\tau_d$ a \emph{sliding block code}.  A function $\phi$ is continuous and commutes with $\sigma$ if and only if $\phi=\tau_d$ for some $d$.  This is usually proved for two-sided shifts as in  \cite{LM95} and it is well-known that is holds for the one-sided case \cite[page 461]{LM95}.  We give a complete proof in \cite[Lemma 3.3.3 and Lemma 3.3.7]{pW10}.

\begin{exam} \label{bar}
Let $A=\{0,1\}$ and define $d:A \to A$ by $d(0)=1$ and $d(1)=0$.  For $x \in A^{\mathbb{N}}$ denote $\tau_d(x)=\overline{x}$.  We know $\tau_d$ commutes with $\sigma$ by definition.  Let $y,z \in A^{\mathbb{N}}$ be such that $\sigma(y)=\tau_d(z)$.  Since $\tau_d$ is bijective, observe that $\overline{y}$ is the unique element in $A^{\mathbb{N}}$ such that $\tau_d(\overline{y})=y$.  We also have that $\tau_d(\sigma(\overline{y}))=\sigma(\tau_d(\overline{y}))= \sigma(y) = \tau_d(z)$ and since $\tau_d$ is bijective $\sigma(\overline{y})=z$.  So $\tau_d$ $*$-commutes with $\sigma$.
\end{exam}

\begin{exam} \label{sigma}
 Let $a_1, a_2  \in A$ such that $a_1 \neq a_2$ and $w \in A^{\mathbb{N}}$.  Observe that $\sigma(a_1 w)=w=\sigma(a_2 w)$ and $a_1 w \neq a_2 w$.  Therefore $\sigma$ does not $*$-commute with itself.  
\end{exam}

Example~\ref{sigma} indicates that the condition $i \neq j$ in Section~\ref{kgraph} is necessary.

Following the standard accepted definitions of permutive in \cite{Hed69} and right permutive in \cite{BK99} we say that the block map $d: A^n \to A$ is \emph{left permutive} if for each fixed $x_1 \cdots x_{n-1} \in A^{n-1}$ the function $r_d^{x_1 \cdots x_{n-1}}: A \to A$ defined by $r_d^{x_1 \cdots x_{n-1}}(a)=d(ax_1 \cdots x_{n-1})$ is bijective.

\begin{exam} \label{counterex}
Let $A = \{0,1,2,3 \}$ and define $d:A^2 \to A$ by
\begin{align*}
&d(00) = 0  &d(01) &= 0  &d(02) &= 1  &d(03) &= 1\\
&d(10) = 3  &d(11) &= 3  &d(12) &= 2  &d(13) &= 2\\
&d(20) = 2  &d(21) &= 2  &d(22) &= 3  &d(23) &= 3\\
&d(30) = 1  &d(31) &= 1  &d(32) &= 0  &d(33) &= 0.
\end{align*}  
The block map $d$ is left permutive.
\end{exam}

\begin{exam} \label{modn}
For this example, all addition is modulo $n$.  Let $A=\{0,1,\cdots, n-1\}$ and define $d:A^n \to A$ by $d(a_1 \cdots a_n)= a_1+ \cdots + a_n \pmod{n}$.  Fix $x_1 \cdots x_{n-1} \in A^{n-1}$ and let $x:= x_1+ \cdots + x_{n-1}$.  To see that $r_d$ is injective, let $a_1, a_2 \in A$ and suppose $r_d^{x_1 \cdots x_{n-1}}(a_1)=r_d^{x_1 \cdots x_{n-1}}(a_2)$.  Then
 \begin{align*}
a_1+x &=d(a_1 x_1 \cdots x_{n-1})= r_d^{x_1 \cdots x_{n-1}}(a_1) \\
           &=r_d^{x_1 \cdots x_{n-1}}(a_2)= d(a_2 x_1 \cdots x_{n-1})=a_2+x,
 \end{align*}
therefore $a_1=a_2$.  Let $a \in A$.  Then we have $r_d^{0_1 \cdots 0_{n-1}}(a)=a$.  Therefore $d$ is left permutive.
\end{exam}

\begin{thm} \label{refiffstar}
The block map $d: A^n \to A$ is left permutive if and only if the induced sliding block code $\tau_d: A^{\mathbb{N}} \to A^{\mathbb{N}}$ $*$-commutes with the shift map $\sigma$.
\end{thm}

\begin{proof}
By definition $\tau_d$ commutes with $\sigma$.  Suppose we have $y,z \in A^{\mathbb{N}}$ such that $\sigma (y) = \tau_d(z)$.  Since $d$ is left permutive there exists a unique $x_1 \in A$ such that $r_d^{z_1 \cdots z_{n-1}}(x_1)= d(x_1 z_1 \cdots z_{n-1})=y_1$.  Notice that $y_{i+1}=\sigma(y)_i= \tau_d(z)_i= \tau_d(x_1 z)_{i+1}$.  So we have
\[
\tau_d(x_1 z)= d(x_1 z_1 \cdots z_{n-1}) \tau_d(x_1 z)_2 \tau_d(x_1 z)_3 \cdots= y_1 y_2 y_3 \cdots= y
\]
and $\sigma(x_1 z)=z$.  To see that $x_1z$ is unique suppose there exists $w \in A^{\mathbb{N}}$ such that $\tau_d(w)=y$ and $\sigma(w)= z$.  Then $w= az$ for some $a \in A$.  Notice that $ d(az_1 \cdots z_{n-1})= \tau_d(az)_1 =y_1=d(x_1z_1 \cdots z_{n-1})$.  Since $d$ is left permutive $a=x_1$.  Therefore $\tau_d$ $*$-commutes with $\sigma$.

Conversely, fix $x_1 \cdots x_{n-1} \in A^{n-1}$.  Suppose for $a_1, a_2 \in A$ we have $r_d^{x_1 \cdots x_{n-1}}(a_1)=r_d^{x_1 \cdots x_{n-1}}(a_2)$.  Then let $z \in Z(x_1 \cdots x_{n-1})$ and observe that
\begin{align*}
\tau_d(a_1 z)_1 &= d(a_1 z_1, \cdots z_{n-1})= r_d^{x_1 \cdots x_{n-1}}(a_1)\\
                          &=r_d^{x_1 \cdots x_{n-1}}(a_2)= d(a_2 z_1, \cdots z_{n-1})= \tau_d(a_2 z)_1.
\end{align*}
For $i \geq 2$ we have $\tau_d(a_1 z)_i= \tau_d(z)_{i-1}= \tau_d(a_2 z)_i$.  So $\tau_d(a_1 z)= \tau_d(a_2 z)$ and $\sigma(a_1 z)=z= \sigma(a_2 z)$. Since $\tau_d$ $*$-commutes with $\sigma$ we have $a_1 z= a_2 z$.  Therefore $r_d^{x_1 \cdots x_{n-1}}$ is injective.  Now let $a \in A$.  Suppose $z \in Z(x_1 \cdots x_{n-1})$ and define $w= \tau_d(z)$.  Then $aw, z \in A^{\mathbb{N}}$ satisfy $\sigma(aw)= \tau_d(z)$.  Since $\tau_d$ and $\sigma$ $*$-commute there exists a unique $v \in A^{\mathbb{N}}$ such that $\sigma(v)=z$ and $\tau_d(v)=aw$.  Since $\sigma(v)=z$, there exists $b \in A$ such that $v=bz$.  So we have $a= \tau_d(v)_1= \tau_d(bz)_1= d(bz_1 \cdots z_{n-1})= d(bx_1 \cdots x_{n-1})$.  So $b \in A$ such that $r_d^{x_1 \cdots x_{n-1}}(b)= d(bx_1 \cdots x_{n-1})= a$.  Therefore $d$ is left permutive.
\end{proof}

Lemma 6.3 of \cite{BK99} gives criteria for a finite-to-one surjective sliding block code $\phi$ to be a local homeomorphism.

\bibliographystyle{amsplain}

\end{document}